\documentclass[11pt]{amsart}
\usepackage{amssymb}
\usepackage{color}
\usepackage{graphicx}
\usepackage{float}
\usepackage[all,cmtip]{xy}
\usepackage{tikz}

\usepackage{tikz-cd}
\usetikzlibrary{matrix}
\usepackage{url}
\usepackage{subfig}
\usepackage{hyperref}
\newcommand{\V}{\operatorname{V}_{4,2}}

\mathchardef\mhyphen="2D

\newcommand{\std}{{\operatorname{std}}}

\newcommand{\Diff}{\operatorname{Diff}}

\newcommand{\SO}{\operatorname{SO}}
\newcommand{\U}{\operatorname{U}}
\newcommand{\SU}{\operatorname{SU}}

\newcommand{\UT}{\operatorname{UT}}

\newcommand{\Id}{{\operatorname{Id}}}

\newcommand{\R}{{\mathbb{R}}}
\newcommand{\C}{{\mathbb{C}}}

\newcommand{\NS}{{\mathbb{S}}}

\newcommand{\Leg}{{\mathfrak{Leg}}}

\newcommand{\FLeg}{{\mathfrak{FLeg}}}

\newcommand{\Emb}{{\mathfrak{Emb}}}

\newtheorem{theorem}{Theorem}
\newtheorem{lemma}[theorem]{Lemma}
\newtheorem{proposition}[theorem]{Proposition}
\newtheorem{corollary}[theorem]{Corollary}

\theoremstyle{definition}
\newtheorem{definition}[theorem]{Definition}

\newtheorem{remark}[theorem]{Remark}

\begin{document} 

    \title{Families of knots that cannot be made Legendrian parametrically}

\date{\today}


\author{Javier Martínez-Aguinaga}

\begin{abstract}

The fact that every smooth knot type admits a Legendrian representative is a classical result in contact topology. However, the analogous surjectivity question was open at the parametric level. In this work we address the $n>1$ case. We prove that for every $n\geq 3$, every knot type $\mathcal K$, every Legendrian representative $\mathcal L$ and every formal Legendrian representative $\mathcal{FL}$, the associated group homomorphisms $\pi_n(\mathcal{L})\to\pi_n(\mathcal{K})$ and $\pi_n(\mathcal{FL})\to\pi_n(\mathcal{K})$ are never surjective. We then show that surjectivity at the $\pi_2$-level depends on the knot type. This work thus proves the presence of rigidity for parametric families at every higher homotopy level beyond $\pi_1$.

\end{abstract}

\maketitle
\setcounter{tocdepth}{1} 

\section{Introduction}

It is a classical result \cite[Thm. 4.4]{FT} that every smooth knot in $\NS^3$ admits a Legendrian representative. This can be understood as a $\pi_0$-surjection between the space of Legendrian embeddings $\Leg(\NS^3,\xi_\std)$ and the space of smooth embeddings $\Emb(\NS^1,\NS^3)$. This showcases that even if the notion of Legendrian knot is a priori more restrictive than the notion of classical knot, every element of the latter type can be isotoped to an element of the former type. However, this flexibility between connected components does not readily translate to the higher-homotopy level. In fact, the analogous question for higher homotopy groups remained open.

T. Kálmán posed this precise question (see \cite[p. 2016]{kalman}) at the level of $\pi_1$; i.e. he asked whether the $\pi_1$-homomorphism induced by the natural inclusion
\begin{equation}\label{NaturalInclusion}
    i:\Leg(\NS^3,\xi_\std) \hookrightarrow \Emb(\NS^1,\NS^3).
\end{equation}

was always surjective. A positive answer to this question has been provided for infinitely many knots in the three main families (hyperbolic, torus and satellites) in the recent preprint \cite{MA} but the question remains open in the general case since first proposed in 2005. 

In this work we address the $n>1$ case, showing that the group homomorphisms induced at the $\pi_n$-level are never surjective for $n\geq3$ and that, nonetheless, the surjectivity for $n=2$ strongly depends on the considered knot types. More precisely, the following results are presented.

\begin{theorem}\label{RigLeg0}    
Let $n\geq 3$. The homotopy group homomorphisms 
\begin{equation} \pi_n\left(\Leg(\NS^3,\xi_\std), \gamma\right)\rightarrow \pi_n\left(\Emb(\NS^1,\NS^3), \gamma\right)
\end{equation}

induced by the natural inclusion (\ref{NaturalInclusion})
are non-surjective for every smooth knot type and every Legendrian representative $\gamma$.
\end{theorem}

Note that this result is sharp in the sense that it is the lowest level where one could expect such a claim to hold. Indeed, the $n=2$-case contrasts with Theorem \ref{RigLeg0}, representing an exceptional case. Analogously, the $\pi_0$-case is flexible and so is the case for infinitely many knots at the $\pi_1$-level \cite{MA}.

\begin{theorem}\label{RigLeg2}    
The homotopy group homomorphism 
\begin{equation} \pi_2\left(\Leg(\NS^3,\xi_\std), \gamma\right)\rightarrow \pi_2\left(\Emb(\NS^1,\NS^3), \gamma\right)
\end{equation}

induced by the natural inclusion (\ref{NaturalInclusion})
is non-surjective whenever $\gamma$ is an unknot and it is trivially surjective if $\gamma$ is either a non-trivial $(p,q)$-torus embedding or a hyperbolic embedding.
\end{theorem}

These results follow from carefully studying the algebraic topology underlying the space of Legendrian embeddings in $(\NS^3,\xi_\std)$ and establishing a comparison with the existing homotopy descriptions of the smooth embedding spaces, extensively studied by A. Hatcher \cite{hatcher2, hatcher3} and by R. Budney \cite{Budney2, Budney3}. 

A completely parallel picture arises when we compare the homotopy type of the space of formal Legendrian knots $\FLeg(\NS^3,\xi_\std)$ with the space of smooth embeddings $\Emb(\NS^1,\NS^3)$. 
The study of homotopical properties of the space of formal Legendrian embeddings has seen much recent activity. This space can be understood as the space of objects that captures the algebraic topology underlying the space of Legendrians. We refer the reader to   \cite{Mur, FMP1}. 

Global topological aspects of the space of formal Legendrians are deeply related to the global topology of the space of Legendrian embeddings. For instance, the homotopy groups of this space can be understood as \textit{formal} invariants in the space of actual Legendrian embeddings. The study of the topology of this space is thus meaningful from the perspective of Legendrian knot spaces. We conclude with the following rigidity result at the higher homotopy-level for the space of formal Legendrian knots.

\begin{theorem} Let $n\geq 3$. The homotopy group homomorphisms 
\begin{equation}\label{map}
\pi_n\left(\FLeg(\NS^3,\xi_\std), (\gamma,F_s)\right)\rightarrow \pi_n\left(\Emb(\NS^1,\NS^3), \gamma\right)
\end{equation}

induced by the forgetful map  $\FLeg(\NS^3,\xi_\std)\rightarrow \Emb(\NS^1,\NS^3)$
are non-surjective for every smooth embedding $\gamma$ and every formal Legendrian representative $(\gamma,F_s)$.
\end{theorem}

\textbf{Acknowledgements:} I would like to thank Mark Grant and Gustavo Granja for useful remarks. The author acknowledges support from PID2022-142024NB
I00 by MICINN (Spain).

\section{Preliminaries on Legendrian embeddings}

\subsection{Standard $3$--dimensional space.}

We will focus our study on the standard contact structures, both on $\R^3$ and on $\NS^3$. These are defined as follows.
The standard contact structure on $\R^3(x,y,z)$ is defined by $\xi_\std=\ker(dz-ydx)$. The standard contact structure on $\NS^3\subseteq \C^2(z_1,z_2)$ is defined as 
        \begin{equation}\label{eq1} \xi_\std =T\NS^3 \cap iT\NS^3=\ker\left(\frac{i}{2}\sum_{j=1}^2 z_jd\bar{z}_j-\bar{z}_j dz_j\right),
        \end{equation}
        where $i:T\C^2\rightarrow T\C^2$ stands for the standard complex structure. Additionally, there exists a well known relationship between these two contact structures.

\begin{remark}\label{contStereo}
    It is well known that $(\R^3,\xi_\std)$ is contactomorphic to $(\NS^3\backslash\{p\},\xi_{\std_{|\NS^3\backslash\{p\}}})$, for any point $p\in\NS^3$ (see \cite[Proposition 2.1.8]{GeigesBook}) via the contact stereographic projection. 
\end{remark}

\subsection{Legendrian embeddings.}

\begin{definition}
	Let $(M,\xi)$ be a contact $3$--manifold. An embedded oriented circle $L\subseteq M$ is said to be a \textbf{Legendrian knot} if it is everywhere tangent to the distribution; i.e. $TL\subseteq\xi$. A \textbf{Legendrian embedding} is any embedding that parametrises a Legendrian knot. 
\end{definition}

Denote by $\widehat{\Leg}(M,\xi)$ the space of Legendrian submanifolds of $(M,\xi)$ and by $\Leg(M,\xi)$ the space of Legendrian embeddings of $(M,\xi)$. Note that $\widehat{\Leg}(M,\xi)=\Leg(M,\xi)/\Diff^+(\NS^1)$, where $\Diff^+(\NS^1)$ stands for the space of orientation-preserving diffeomorphisms of $\NS^1$.

\subsection{Smooth embeddings in $\NS^3$ and long embeddings in $\R^3$.}

\begin{definition}\label{LongEmbeddings}
    Write $\Emb_{N,jN}(\NS^1,\NS^3)$ for the space of smooth embeddings $\gamma(t)$ in $\NS^3$ such that at $t=0$  they pass through the north pole $N=(1,0)\in\C^2$ with fixed derivative $j\cdot N=(0,1)\in\C^2$ (where we are using quaternionic notation). We will call this space the space of long knots. Precisely:
$$\Emb_{N,jN}(\NS^1,\NS^3)=\{\gamma\in\Emb(\NS^1,\NS^3):\gamma(0)=N,\gamma'(0)=jN\}.$$ 
\end{definition}

\begin{remark}\label{aspherical} Every component of the space $\Emb_{N,jN}(\NS^1,\NS^3)$) is a $K(\pi, 1)$-space \cite[Thm 1.2(1)]{BudCohen}. This follows from the work of A. Hatcher \cite{hatcher2, hatcher3}. Equivalently, we have that $\pi_n\left(\Emb_{N,jN}(\NS^1,\NS^3)\right)=\{0\}$ for $n\geq 2$. 
\end{remark}

\begin{remark}[Notation]\label{notation}
Henceforth, when we write a superscript $\gamma$ over any embedding space in this article, we are denoting the connected component of that space containing 
 the embedding $\gamma$. For instance, $\Emb^\gamma(\NS^1,\NS^3)$ denotes the connected component of the space of smooth embeddings $\Emb(\NS^1,\NS^3)$ containing the embedding $\gamma$.
\end{remark}

Following \cite{Budney2, Budney3, BudCohen}, we consider the following  smooth fibration
\begin{equation}\label{eq:LongEmbeddingsFibration}
\Emb_{N,jN}(\NS^1,\NS^3)\hookrightarrow\Emb(\NS^1,\NS^3)\rightarrow \tilde{\operatorname{V}}_{4,2},
\end{equation}
where $\tilde{\operatorname{V}}_{4,2}:=\V\times(0,\infty)$ and $\V$ denotes the Stiefel manifold of $2$-frames in $\R^4$. The fibration map projects an embedding $\gamma\in\Emb(\NS^1,\NS^3)$ onto the element $\left((\gamma(0),\frac{\gamma'(0)}{||\gamma'(0)||}),||\gamma'(0)||\right)\in \tilde{\operatorname{V}}_{4,2}=\V\times(0,\infty) \cong \SO(4)/\SO(2) \times \R$. The fiber over $\left((N, jN), 1\right)\in\tilde{\operatorname{V}}_{4,2}$ is the space $\Emb_{N,jN}(\NS^1,\NS^3)$ and all the fibers are homotopically equivalent. The $(0,\infty)$-factor in the base $\tilde{\operatorname{V}}_{4,2}$ just codifies the norm of the projected derivative but note that, up to homotopy, this is completely irrelevant by the contractibility of $\R$.

This fibration relates the topology of the space of smooth embeddings in $\NS^3$ with the topology of the space of long embeddings.

\begin{remark}\label{StiefelS3S2}
    The topology of the Stiefel manifold $\V$; equivalently, $SO(4)/SO(2)$, is well understood. This manifold coincides with $\UT\NS^3$, the unit tangent bundle of $\NS^3$. Furthermore, note that since $\NS^3$ is parallelisable, we can identify $\UT\NS^3$ with $\NS^3\times\NS^2$. Thus, we have the homotopy equivalences $\tilde{\operatorname{V}}_{4,2}\cong \V\cong\SO(4)/\SO(2)\cong \NS^3\times\NS^2$.
\end{remark}

    \begin{remark} The group $\SO(4)$ acting on $\NS^3$ naturally acts on the space of embeddings by restriction. This is the $\SO(4)$ from the quotient in the base of Fibration (\ref{eq:LongEmbeddingsFibration}). Finally, the $\SO(2)$ in the quotient of Fibration (\ref{eq:LongEmbeddingsFibration})  identifies with the stabilizer within $\SO(4)$ of the $2$-frame $(N,jN)\in\V$.
    \end{remark}

\subsection{Legendrian embeddings and long Legendrians}
We have an analogous situation in the contact setting, as already pointed out in \cite{FMP2, FM}.

The homotopy type of the space of Legendrian embeddings in the standard $(\NS^3,\xi_\std)$ is intimately related to the space of long Legendrian embeddings $\Leg_{N,jN}(\NS^3,\xi_\std):=\Emb_{N,jN}(\NS^1,\NS^3)\cap\Leg(\NS^3,\xi_\std).$ 

\begin{lemma}\label{normalisation}
    The space $\Leg(\NS^3,\xi_\std)$ is homotopy equivalent to ${\Leg}_*(\NS^3,\xi_\std):=\{\gamma\in\Leg(\NS^3,\xi_\std):||\gamma'(0)||=1\}$.
\end{lemma}
\begin{proof}

Identify $\NS^1$ with the interval $[-1,1]$ with identified endpoints. Take a smooth family of diffeomorphisms $\{\phi_r\}_{r>0}:\NS^1\to\NS^1$ such that\footnote{This family always exists. For instance, take any family of functions $f_r:[-1,1]\to\R^+$ such that: i) $f_r(t)=1$ for $t\notin[\frac{-1}{4},\frac{1}{4}]$, ii) $f_r(0)=\frac{1}{r}$ and iii) $\int_{\frac{-1}{4}}^0 f_r(s)ds=\int_{0}^\frac{1}{4} f_r(s)ds=\frac{1}{4}$ and just define $\phi_r(t):=\frac{-1}{4}+\int_{\frac{-1}{4}}^tf_r(s)ds$. Finally identify $[-1,1]_{-1\sim1}$ with $\NS^1$.}:
\[ \phi_r(t)=t \text{ for }t\notin\left(\frac{-1}{4},\frac{1}{4}\right),\quad\phi_r(0)=0\quad\text{ and }\quad||\phi'_r(0)||=\frac{1}{r}.\]
Write $r_\gamma:=||\gamma'(0)||$ and define the following map:
 \begin{equation}\label{projection}
	\begin{array}{rccl}
\varphi:\Leg(\NS^3,\xi_\std)&\longrightarrow&\Leg_*(\NS^3,\xi_\std)\\
	\gamma(t) & \mapsto &\tilde{\gamma}(t):=\gamma(\phi_{r_\gamma}(t)).
	\end{array}
	\end{equation}
    Consider the natural inclusion $\iota:\Leg_*(\NS^3,\xi_\std)\to\Leg(\NS^3,\xi_\std)$ and define the following homotopy in $\Leg(\NS^3,\xi_\std)$: $\Gamma_s(\gamma(t)):=\gamma((1-s)\cdot t+s\cdot\phi_{r_\gamma}(t))$, $s\in[0,1]$. It readily follows that it restricts to a homotopy within $\Leg_*(\NS^3,\xi_\std)$. By means of this homotopy we conclude  $\varphi\circ\iota\cong \Id|_{\Leg_*(\NS^3,\xi_\std)}$ and $\iota\circ\varphi\cong \Id|_{\Leg(\NS^3,\xi_\std)}$, thus yielding the claim.\end{proof}

There is a homotopical decomposition of the space of Legendrians (see \cite[Sec. 7.4]{FMP2}) that we will introduce in the following proposition. 

\begin{proposition}\label{lem:HomotopyEquivalence}
	The space of Legendrian embeddings in  $(S^3,\xi_\std)$ admits the following homotopy equivalence:
\begin{equation}\label{eq:HomotopyLeg}
	\Leg(\NS^3,\xi_\std)\cong \U(2)\times \Leg_{N,jN}(\NS^3,\xi_\std).
    \end{equation}
\end{proposition}
\begin{proof} Up to homotopy, we can assume that the elements $\gamma\in\Leg(\NS^3,\xi_\std)$ satisfy $||\gamma'(0)||=1$ by Lemma \ref{normalisation}. Consider the following natural projection:
 \begin{equation}\label{projection}
	\begin{array}{rccl}
\rho:\Leg(\NS^3,\xi_\std)&\longrightarrow&\U(2)\\
	\gamma & \mapsto &A_\gamma:=\left(\gamma(0) \phantom{,}|\phantom{,} \gamma'(0)\right).
	\end{array}
	\end{equation}

 Note that $A_\gamma:=\left(\gamma(0) \phantom{,}|\phantom{,} \gamma'(0)\right)$ lies in $\U(2)$ because $\xi_\std$ is defined as the distribution given by the complex tangencies (Equation (\ref{eq1})). 
 
 Finally, consider the following homotopy equivalence that does the work:
\begin{equation}\label{eq:HomotopyDescompLegS3}
	\begin{array}{rccl}
\Phi:\Leg(\NS^3,\xi_\std)&\longrightarrow&\U(2)\times \Leg_{N,jN}(\NS^3,\xi_\std)\\
	\gamma & \mapsto &(A_\gamma,\, 
 A_{\gamma}^{-1}\cdot\gamma).
	\end{array}
	\end{equation} 

In order to conclude, just note that the following map provides a homotopy inverse:
\begin{equation}\label{eq:HomotopyDescompLegS3Inverse}
	\begin{array}{rccl}
\Psi:\U(2)\times \Leg_{N,jN}(\NS^3,\xi_\std)&\longrightarrow&\Leg(\NS^3,\xi_\std)\\
	(A,\beta) & \mapsto &A\cdot\beta.
	\end{array}
	\end{equation} \end{proof}

\begin{remark}
    By means of Proposition \ref{lem:HomotopyEquivalence}, the embedding $\gamma \in\Leg(\NS^3,\xi_\std)$ has a naturally associated long embedding  $\gamma_\ell:=A^{-1}_\gamma\cdot\gamma\in\Leg_{N,jN}(\NS^3,\xi_\std)$. This will often be used in the notation. 
\end{remark}

\subsection{Homotopy-group homomorphisms}

The homotopical decomposition of the space of Legendrian embeddings in $(\NS^3,\xi_\std)$ from Proposition \ref{lem:HomotopyEquivalence} induces the following decomposition at the level of homotopy groups.
\begin{equation}\label{eq:HomotopyLeg1}
	\pi_n\left(\Leg(\NS^3,\xi_\std),\gamma\right) = \pi_n\left(\U(2),A_\gamma\right)\oplus \pi_n\left(\Leg_{N,jN}(\NS^3,\xi_\std),\gamma_{\ell}\right).
	\end{equation} 
    
We can consider the following trivial fibration via the projection $\rho$ from  Expression (\ref{projection}):
\begin{equation}\label{eq:LongLegendriansFibration}
\Leg_{N,jN}(\NS^3,\xi_\std)\hookrightarrow\Leg(\NS^3,\xi_\std)\rightarrow \U(2).
\end{equation}
If we assume that, up to homotopy, the elements  $\gamma\in\Leg(\NS^3,\xi_\std)$ satisfy $||\gamma'(0)||=1$ (Recall Lemma \ref{normalisation}) then this fibration is naturally related to  Fibration (\ref{eq:LongEmbeddingsFibration}). Indeed, if we consider the inclusion $\iota:\U(2)\hookrightarrow\tilde{\operatorname{V}}_{4,2},\,A\mapsto(A,1)$, then we get the following commutative diagram at the level of homotopy groups:
\[
\begin{tikzcd}\label{conmutativo}
\pi_n\left(\Leg_{N,jN}(\NS^3,\xi_\std),\gamma_\ell\right) \arrow[r] \arrow[d]
  & \pi_n\left(\Leg(\NS^3,\xi_\std),\gamma\right)  \arrow[r] \arrow[d] & \pi_n(\U(2),A_\gamma) \arrow[d]\\
\pi_n\left(\Emb_{N,jN}(\NS^1,\NS^3),\gamma_\ell\right) \arrow[r]
  & \pi_n\left(\Emb(\NS^1,\NS^3),\gamma\right)\arrow[r] & \pi_n(\tilde{\operatorname{V}}_{4,2})
\end{tikzcd}
\]

In particular, we have that the group homomorphisms
\begin{equation}\label{PInMorphism}
\pi_n\left(\Leg(\NS^3,\xi_\std),\gamma\right)\rightarrow\pi_n\left(\Emb(\NS^1,\NS^3),\gamma\right)
\end{equation}

yield, by restriction to the long components, the group homomorphisms
\begin{equation}\label{LongDec}
\pi_n\left(\Leg_{N,jN}(\NS^3,\xi_\std),\gamma_\ell\right)\rightarrow\pi_n\left(\Emb_{N,jN}(\NS^1,\NS^3),\gamma_\ell\right).
\end{equation}

\begin{remark}
   Following Remark \ref{notation}, we will write $\Leg^\gamma(\NS^3, \xi_\std)$ and $\Leg^\gamma_{N,jN}(\NS^3, \xi_\std)$ to denote the connected component containing a Legendrian embedding $\gamma\in\Leg(\NS^3, \xi_\std)$ (respectively, long Legendrian embedding $\gamma\in \Leg_{N,jN}(\NS^3, \xi_\std)$).
\end{remark}

Notice that the long smooth components are aspherical (Remark \ref{aspherical}), i.e. the homotopy groups $\pi_n\left(\Emb_{N,jN}(\NS^1,\NS^3),\gamma_\ell\right)$ are trivial for $n\geq 2$ (and thus have no contribution onto their long smooth part). Therefore, this implies the following result.

\begin{proposition} Let $\gamma\in\Leg(\NS^3,\xi_\std)$ and $n\geq 2$. The $\pi_n$-homomorphisms induced by the natural inclusion $\pi_n\left(\Leg(\NS^3,\xi_\std), \gamma\right)\rightarrow \pi_n\left(\Emb(\NS^1,\NS^3), \gamma\right)$ have exactly the same image as the associated group homomorphisms $\pi_n\left(\U(2), A_\gamma\right)\rightarrow\pi_n\left(\Emb(\NS^1,\NS^3),\gamma\right)$ 
induced by the inclusion of the first summand in Decomposition (\ref{eq:HomotopyLeg1}) into $\Leg(\NS^3,\xi_\std)$. 
\end{proposition}

In particular, we get the following corollary.

\begin{corollary}\label{PropSurj2} Let $\gamma\in\Leg(\NS^3,\xi_\std)$. The $\pi_n$-homomorphisms 
\begin{equation} \pi_n\left(\Leg(\NS^3,\xi_\std), \gamma\right)\rightarrow \pi_n\left(\Emb(\NS^1,\NS^3), \gamma\right)
\end{equation}
induced by the natural inclusion are surjective for $n\geq 2$ if and only if the  group homomorphisms $\pi_n\left(\U(2),A_\gamma\right)\rightarrow\pi_n\left(\Emb(\NS^1,\NS^3),\gamma\right)$ given by restriction to the first summand in Decomposition (\ref{eq:HomotopyLeg1}) are surjective for $n\geq 2$.
\end{corollary}

\section{Rigidity for higher homotopy groups}

The  $\pi_n$-homotopy group homomorphisms induced by the natural inclusion (\ref{NaturalInclusion}) provide, altogether, information about the global topology of the smooth and Legendrian embedding spaces, respectively, and how they relate to each other. We will show the lack of surjectivity for homotopy groups beyond $n=2$. The $n=2$ case will be treated separately. 

    \subsection{Surjectivity at the $\pi_2$-level} Surjectivity at the $\pi_2$-level holds trivially for some  components. For instance, it follows from Theorem \ref{thm:HatcherTorusKnots}  below that $\pi_2(Emb^\gamma(\NS^1,\NS^3))$ is the trivial group if $\gamma$ is either a hyperbolic embedding or a  non-trivial $(p,q)$-torus embedding and, thus, surjectivity follows trivially. On the contrary, note that if $\gamma$ is the unknot, then there is no $\pi_2$-surjectivity  since $\pi_2(Emb^\gamma(\NS^1,\NS^3),\gamma)=\pi_2\left(\SO(4)/\SO(2)\right)=\mathbb{Z}$, whereas $\pi_2(\U(2))=\{0\}$ (and the claim follows by Corollary \ref{PropSurj2}). This readily proves Theorem \ref{RigLeg2} from the Introduction.

Following the exposition from \cite[Sec 1.1]{HK}, we focus our attention on the following result that combines several celebrated results from the work of A. Hatcher and which is well known in the theory of smooth embedding spaces.

\begin{theorem}\label{thm:HatcherTorusKnots} (Hatcher, \cite{hatcher2, hatcher3}, see also \cite[Sec 1.1]{HK}). Take a smooth embedding $\gamma\in\Emb^{\gamma}(\NS^1,\NS^3)$ and an associated long embedding $\gamma_\ell\in\Emb^{\gamma}(\NS^1,\NS^3)$. The following homotopy equivalences hold.
   \begin{itemize}
   
       \item[i)] If $\gamma$ is the unknot, then $\Emb^{\gamma}(\NS^1,\NS^3)\cong \V$ and $\Emb_{N,jN}^{\gamma_\ell}(\NS^1,\NS^3)$ is contractible.
       \item[ii)] If $\gamma$ is a non-trivial $(p,q)$-torus embedding, $\Emb^{\gamma}(\NS^1,\NS^3)\cong\SO(4)$ and $\Emb_{N,jN}^{\gamma_\ell}(\NS^1,\NS^3)\cong\NS^1$.
       \item[iii)] If $\gamma$ is a hyperbolic, $\Emb^{\gamma}(\NS^1,\NS^3)\cong\SO(4)\times\NS^1$ and $\Emb_{N,jN}^{\gamma_\ell}(\NS^1,\NS^3)\cong\NS^1\times\NS^1$.
\end{itemize}

 \end{theorem}

\subsection{The $n\geq 3$ case}

We will devote the rest of the section to tackling the $n\geq 3$ case.

The fibration (\ref{eq:LongEmbeddingsFibration})
$\Emb_{N,jN}(\NS^1,\NS^3)\hookrightarrow\Emb(\NS^1,\NS^3)\rightarrow \tilde{\operatorname{V}}_{4,2}\cong\SO(4)/\SO(2)$
induces a long exact sequence at the level of homotopy groups: \begin{displaymath}
	\xymatrix@M=3pt{
		& \cdots\ar[r] & \pi_{n+1} (\SO(4)/\SO(2))\ar[dll]\\
		\pi_n (\Emb_{N,jN}(\NS^1,\NS^3),\gamma_\ell) \ar[r] & \pi_n (\Emb(\NS^1,\NS^3),\gamma) \ar[r] & \pi_{n}(\SO(4)/\SO(2))\ar[dll] \\
		\pi_{n-1} (\Emb_{N,jN}(\NS^1,\NS^3),\gamma_\ell) \ar[r] & \cdots}
	\end{displaymath}
    
The following sequences are exact for $n\geq 3$ by taking into account Remark \ref{aspherical}.
\begin{equation}
	\xymatrix@M=10pt{
		0\ar[r] & \pi_n (\Emb(\NS^1,\NS^3),\gamma)\ar[r] & \pi_{n}(\SO(4)/\SO(2)) \ar[r] & 0,} \quad n\geq 3.
	\label{ExactSequence}
\end{equation}
This readily yields the following result.
\begin{proposition}\label{HigherHomotopy}
    The following group isomorphisms hold
    for every $n\geq 3$.   
    \begin{equation}\label{isomorfismoNudosSO4}
        \pi_n (\Emb^\gamma(\NS^1,\NS^3),\gamma)\cong \pi_{n}(\SO(4)/\SO(2)).
    \end{equation}
    
\end{proposition}

 Combining Proposition 
\ref{HigherHomotopy} and Remark \ref{StiefelS3S2},
we conclude the following lemma.
\begin{lemma}\label{lemma:Clave_pin} Let $n\geq 3$. Then the following group isomorphisms hold: \[\pi_n\left(\Emb(\NS^1,\NS^3),\gamma\right)\cong \pi_n(\SO(4)/\SO(2)\times(0,\infty))\cong\pi_n(\SO(4)/\SO(2))\cong \pi_n(\V)\cong \pi_n(\NS^3)\oplus\pi_n(\NS^2).\]
where the first isomorphism is induced by the map from fibration (\ref{eq:LongEmbeddingsFibration}) and the second isomorphism is simply induced by the projection $\SO(4)/\SO(2)\times(0,\infty)\to\SO(4)/\SO(2)$.
\end{lemma}

The following lemma is an algebraic result that will be key in order to prove our main theorem.

\begin{lemma}\label{LackSurjectivity}
    Every group homomorphism $\rho: \pi_n\left(\U(2)\right) \to \pi_n(\NS^3)\oplus\pi_n(\NS^2)$ is non-surjective if $n\geq 3$.
\end{lemma}
\begin{proof}

First, note that the determinant $\det:\U(m)\to\NS^1$ yields a fibration $\SU(m)\hookrightarrow\U(m)\to\NS^1$ for every $m\geq 2$ and, in particular, for our case of interest $m=2$. By considering the associated long exact sequence of homotopy groups, we conclude that $\pi_n(\U(2))\cong \pi_n(\SU(2))$ for $n\geq 2$.

On the other hand, $\SU(2)$ is diffeomorphic to $\NS^3$. Therefore, $\pi_n(\U(2))\cong \pi_n(\NS^3)$ for $n\geq 2$ and so the proof boils down to showing that there cannot exist any group epimorphism $f:\pi_n\left(\NS^3\right)\to
\pi_n\left(\NS^3\right)\oplus\pi_n\left(\NS^2\right)$ for $n\geq 3$.

The Hopf fibration $\NS^1\hookrightarrow \NS^3 \to \NS^2$ yields a long exact sequence of homotopy groups that allows us to conclude that $\pi_n(\NS^2)\cong\pi_n(\NS^3)$ for $n\geq 3$. Therefore, if we show that there cannot exist a group epimorphism $\tau: \pi_n(\NS^2)\to\pi_n(\NS^2)\oplus\pi_n(\NS^2)$ for $n\geq 3$ we will be done. 

We know that homotopy groups of spheres are abelian and finitely generated \cite{Serre}. On the other hand, all higher homotopy groups of $\NS^2$ are non-trivial by the work  \cite{IMW} of S. O. Ivanov, R. Mikhailov and J. Wu; i.e. $\pi_n(\NS^2)\neq\{0\}$ for $n\geq 2$. 

Finally, note that it is not possible to have a group epimorphism between a non-trivial abelian, finitely generated group $G$ and $G\oplus G$. Indeed, if $G$ is purely torsion, the claim follows purely by cardinality reasons. If $G$ has a non-zero free component, then the lack of surjectivity readily follows because of rank differences. Therefore, the claim follows.
\end{proof}

We finally conclude this subsection with Theorem \ref{RigLeg}. This will follow from all the arguments and constructions developed until this point.

\begin{theorem}\label{RigLeg}    
Let $n\geq 3$. The homotopy group homomorphisms 
\begin{equation} \pi_n\left(\Leg(\NS^3,\xi_\std), \gamma\right)\rightarrow \pi_n\left(\Emb(\NS^1,\NS^3), \gamma\right)
\end{equation}

induced by the natural inclusion (\ref{NaturalInclusion})
are non-surjective for every smooth knot type and every Legendrian representative $\gamma$.
\end{theorem}

\begin{proof}
By Corollary \ref{PropSurj2}, the claim boils down to checking that the morphisms $\pi_n\left(\U(2)\right)\rightarrow\pi_n\left(\Emb(\NS^1,\NS^3)\right)$ cannot be surjective for $n\geq 3$. Since $\pi_n (\Emb^\gamma(\NS^1,\NS^3),\gamma)\cong \pi_{n}(\V)\cong\pi_n(\NS^3)\oplus\pi_n(\NS^2)$ by Proposition \ref{HigherHomotopy} and Lemma \ref{lemma:Clave_pin}, then the claim follows by Lemma \ref{LackSurjectivity}.
\end{proof}

\subsection{The formal case}
Let us introduce Formal Legendrian embeddings first.

\begin{definition}
    A \textbf{formal Legendrian embedding} in $(\NS^3,\xi_\std)$ is a pair $(\gamma,F_s)$ satisfying: 
		\begin{itemize}
			\item [(i)] $\gamma:\NS^1\rightarrow\NS^3$ is an embedding.
			\item [(ii)] $F_s:\NS^1\rightarrow \gamma^*(T\NS^3\setminus\lbrace 0\rbrace)$ is a $1$--parametric smooth family of vector fields along $\gamma$, $s\in[0,1]$, such that $F_0=\gamma'$ and $F_1(t)\in\xi_{\gamma(t)}$.
		\end{itemize}
  We denote the space of formal Legendrian embeddings in $(\NS^3,\xi_\std)$ by $\FLeg(\NS^3,\xi_\std)$.
\end{definition}

Note that there is a natural forgetful map 
\begin{equation}\label{eq:forgetful}
\FLeg(\NS^3,\xi_\std)\rightarrow\Emb(\NS^1,\NS^3),\quad (\gamma,F_s)\mapsto \gamma
\end{equation}
that induces group homomorphisms
$\pi_n\left(\FLeg(\NS^3,\xi_\std),(\gamma,F_s)\right)\to\pi_n\left(\Emb(\NS^1,\NS^3),\gamma\right)$.

We will prove that these homomorphisms are never surjective for $n\geq 3$. Let us introduce the notion of simple space first that will turn out to be useful later, as well as two useful remarks.

\begin{definition}\label{simple}\cite[Def. 1.4.4]{MP}
    We say that a connected topological space $X$ is a \textbf{simple space} if its fundamental group $\pi_1(X)$ is abelian and acts trivially on all higher homotopy groups $\pi_n(X)$, $n\geq 2$.
\end{definition}

\begin{remark}\label{Hspace}
 The set of free homotopy classes of maps from $\NS^n$ to a connected topological space $X$, denoted by $[\NS^n,X]$, is in bijection with $\pi_n(X)/\pi_1(X)$, the set of orbits of the action of $\pi_1(X)$ on $\pi_n(X)$. In particular, if $X$ is a simple space, then $[\NS^n,X]$ is actually a group isomorphic to $\pi_n(X)$ (\cite[Lemma 1.4.2]{MP}). Examples of simple spaces are $H$-spaces and in particular connected Lie groups \cite[Prop. 1.4.3]{MP}, such as $\U(2)$, or simply connected spaces (since both conditions in Definition \ref{simple} are trivially satisfied), such as $\V\cong \NS^3\times\NS^2$. This will be useful later. 
\end{remark}

Another key remark that will be useful in the upcoming argument reads as follows.

\begin{remark}\label{MeasureOrigin}
Given $n\geq 3$ and a smooth $n$-sphere $\{\gamma^k\}_{k\in\NS^n}$ of embeddings based at $\gamma$, note that, by Lemma \ref{lemma:Clave_pin}, the smooth class $A=[\{\gamma^k\}_{k\in\NS^n}]\in\pi_n\left(\Emb(\NS^1,\NS^3),\gamma\right)$ represented by this sphere within $\pi_n\left(\Emb(\NS^1,\NS^3),\gamma\right)$ only depends on the evaluations of $\{\gamma^k\}_{k\in\NS^n}$ and its derivative at the origin; i.e. on the families of values $\{\gamma^k(0)\}_{k\in\NS^n}$ and $\{(\gamma^k)'(0)\}_{k\in\NS^n}$. In particular, by the isomorphism $\pi_n(\Emb(\NS^1,\NS^3),\gamma)\cong\pi_n(\V)$ given in Lemma \ref{lemma:Clave_pin}, we have that the smooth class $A\in\pi_n\left(\Emb(\NS^1,\NS^3),\gamma\right)$ is represented by $[\lbrace(\gamma^k(0),\frac{(\gamma^k)'(0)}{||(\gamma^k)'(0)||})\rbrace_{k\in\NS^n}]\in\pi_n(\V)$.
\end{remark}

Let us state the main theorem of this section and provide a proof taking into account all the tools at our disposal.

\begin{theorem}\label{t28} Let $n\geq 3$. The homotopy group homomorphisms 
\begin{equation}\label{map}
\pi_n\left(\FLeg(\NS^3,\xi_\std), (\gamma,F_s)\right)\rightarrow \pi_n\left(\Emb(\NS^1,\NS^3), \gamma\right)
\end{equation}

induced by the forgetful map  $\FLeg(\NS^3,\xi_\std)\rightarrow \Emb(\NS^1,\NS^3)$
are non-surjective for every smooth embedding $\gamma$ and every formal Legendrian representative $(\gamma,F_s)$.
\end{theorem}
\begin{proof}

Fix $n\geq 3$ and take an arbitrary smooth class $[A]\in\pi_n\left(\Emb(\NS^1,\NS^3),\gamma\right)\cong\pi_n(\V)$. Assume that the $\pi_n$-homomorphism (\ref{map}) induced by the forgetful map (\ref{eq:forgetful}) is surjective and let us derive a contradiction. Assume thus that there exists an $n$-sphere of formal Legendrians $\{(\gamma^k,F_s^k)\}_{k\in\NS^n}$ based at $(\gamma,F_s)\in\FLeg(\NS^3,\xi_\std)$ mapping to the class $[A]\in\pi_n\left(\Emb(\NS^1,\NS^3),\gamma\right)$ by the $\pi_n$-homomorphism (\ref{map}).

Recall that we can measure its smooth $\pi_n$-class at $t=0$ (Remark \ref{MeasureOrigin}). More precisely, if we consider the family of matrices $\{A^k\}_{k\in\NS^n}$ where $A^k=\left(\gamma^k(0)||\frac{(\gamma^k)'(0))}{||(\gamma^k)'(0))||}\right)$, then the smooth class of the sphere $\{(\gamma^k,F_s^k)\}_{k\in\NS^n}$ is precisely given by $[\{A^k\}_{k\in\NS^n}]\in\pi_n(\V)\cong\pi_n(\Emb(\NS^1,\NS^3),\gamma)$ (recall Remark \ref{MeasureOrigin}).

Regard $\{A^k\}_{k\in\NS^n}$ as a non-based sphere within the space $\V$ and thus representing a free homotopy class $[\{A^k\}_{k\in\NS^n}]\in[\NS^n,\V]$. By using the formal legendrian derivative we can (freely) homotope $\{A^k\}_{k\in\NS^n}$ to a new sphere that lies entirely in $\U(2)$. Indeed, consider the following explicit free homotopy of spheres:
\begin{equation}
    B_{s}:=\left\{\left(\gamma^k(0)\,|\,\frac{F^k_s(0)}{||F_s^k(0)||}\right) \right\}_{k\in\NS^n},\quad s\in[0,1].
\end{equation}
It connects the original sphere \[B_0=\left\{\left(\gamma^k(0)\,|\,\frac{F^k_0(0)}{||F_0^k(0)||}\right) \right\}_{k\in\NS^n}=\left\{\left(\gamma^k(0)\,|\,\frac{(\gamma^k)'(0)}{||(\gamma^k)'(0)||}\right) \right\}_{k\in\NS^n}\] with $B_1=\left\{\left(\gamma^k(0)\,|\,\frac{F^k_1(0)}{||F_1^k(0)||}\right) \right\}_{k\in\NS^n}\subset \U(2)$. The fact that the latter sphere lies entirely in $\U(2)$ follows from the fact that $F^k_1(0)\in(\xi_\std)_{\gamma^k(0)}$ and from the fact that $\xi_\std$ is defined as the distribution given by the complex tangencies (Eq. (\ref{eq1})). Consequently, the latter sphere represents a free homotopy class in $\U(2)$, i.e. $[\{(\gamma^k(0)\,|\,\frac{F^k_1(0)}{||F^k_1(0)||})\}]\in[\NS^n,\U(2)]$. Note that the homotopy $B_s$, $s\in[0,1]$ does not necessarily preserve the basepoints. For that reason we work with free homotopies instead, since by Remark \ref{Hspace} we obtain the commutative diagram from Figure \ref{conmutativo}.

We have shown that given a class $[A]\in\pi_n(\V)\cong[\NS^n,\V]$ (Remark \ref{Hspace}), there exists a free homotopy class $[\tilde{A}]\in[\NS^n,\U(2)]$ that represents the same free homotopy class $[A]\in[\NS^n,\V]$ under the inclusion map $i:\U(2)\hookrightarrow\V$. Since both $\U(2)$ and $\V$ are simple spaces, then Remark \ref{Hspace} implies that for $n\geq 3$ we have a surjective group homomorphism from  $\pi_n(\U(2))$ to $ \pi_n(\V)\cong \pi_n(\NS^3)\oplus \pi_n(\NS^2)$, contradicting Lemma \ref{LackSurjectivity}. Indeed, see the diagram in Figure \ref{conmutativo}. The upper horizontal arrow corresponds to the $\pi_n$-homomorphism induced by the inclusion. The vertical arrows are group isomorphisms that follow from Remark \ref{Hspace}, since $\U(2)$ and $\V$ are simple topological spaces. As a consequence, the surjectivity of the lower horizontal morphism implies the existence of a surjective group homomorphism upstairs. 
\begin{figure}[h!]
    \centering
   \[
\begin{tikzcd}
\pi_n\left(\U(2)\right) \arrow[d] \arrow[r]
  & \pi_n\left(\V\right) \arrow[d] \\
{[\NS^n, \U(2)]} \arrow[r] 
  & {[\NS^n,\V]}
\end{tikzcd}
\]
    \caption{Commutative diagram depicting part of the argument in the proof of Theorem \ref{t28}.}
    \label{conmutativo}
\end{figure}

The contradiction arises from assuming that the map $\pi_n\left(\FLeg(\NS^3,\xi_\std), (\gamma,F_s)\right)\rightarrow \pi_n\left(\Emb(\NS^1,\NS^3), \gamma\right)$ was surjective, thus yielding the claim.
\end{proof}

\textbf{Use of AI tools}. The author used
OpenAI's ChatGPT 5.6 model (accessed July and August 2026) as an auxiliary tool for proofreading and identifying errors, typos and points requiring further clarification in the article.

\end{document}